\documentclass[10pt]{amsart}
\usepackage[english]{babel}
\usepackage{amsmath,amssymb,amsthm,amscd}
\usepackage{url}
\usepackage{color,hyperref}

\usepackage{tikz}
\usetikzlibrary{fit}

\theoremstyle{plain}
\newtheorem{theorem}{Theorem}[section]
\newtheorem{proposition}[theorem]{Proposition}
\newtheorem{lemma}[theorem]{Lemma}
\newtheorem{corollary}[theorem]{Corollary}

\theoremstyle{definition}

\newtheorem{remark}[theorem]{Remark}
\newtheorem{example}[theorem]{Example}

\newtheorem*{conjecture}{Conjecture}

\numberwithin{equation}{section}

\newcommand{\N}{\mathbb{N}}

\newcommand{\R}{\mathbb{R}}
\newcommand{\set}[2]{\{#1\,|\ #2\}}
\newcommand{\sub}{\subseteq}
\newcommand{\tabulka}[9]{
\begin{tabular}{c}
#1\\
\begin{tabular}{ccc}
\begin{tabular}[t]{r|cc}%
$+$ & $a$ & $b$ \\\hline
$a$ & #2 & #3 \\
$b$ & #4 & #5\\
\end{tabular}
&&
\begin{tabular}[t]{r|cc}%
$\cdot$ & $a$ & $b$ \\\hline
$a$ & #6 & #7 \\
$b$ & #8 & #9\\
\end{tabular}
\end{tabular}
\end{tabular}
}

\makeatletter
\@namedef{subjclassname@2020}{\textup{2020} Mathematics Subject Classification}
\makeatother

\begin{document}
\title[Congruence-simple semirings without nilpotent elements]{Congruence-simple semirings without nilpotent elements}

 \author[T.~Kepka]{Tom\'{a}\v{s}~Kepka}
 \address{Department of Algebra, Faculty of Mathematics and Physics, Charles University, Sokolovsk\'{a} 83, 186 75 Prague 8, Czech Republic}
 \email{kepka@karlin.mff.cuni.cz}

\author[M.~Korbel\'a\v{r}]{Miroslav~Korbel\'a\v{r}}
\address{Department of Mathematics, Faculty of Electrical Engineering, Czech Technical University in Prague, Technick\'{a} 2, 166 27 Prague 6, Czech Republic}
\email{korbemir@fel.cvut.cz}

\author[G.~Landsmann]{\textsc{G\"{u}nter Landsmann}}
\address{Research Institute for Symbolic Computation, Johannes Kepler University, Alten\-bergerstr. 69, A-4040 Linz, Austria}
\email{landsmann@risc.jku.at}

\thanks{
The second and third authors acknowledge the support by the bilateral Austrian Science Fund (FWF) project I 4579-N and Czech Science Foundation (GA\v CR) project 20-09869L "``The many facets of orthomodularity''}

\date{\today}


\begin{abstract}
We provide a classification of congruence-simple semirings with a multiplicatively absorbing element and without non-trivial nilpotent elements.
\end{abstract}

\maketitle

Semirings gained a lot of attention in recent time thanks to their applications in various branches of mathematics (for more details see \cite{golan,weinert}). Congruence-simple semirings, as the structural keystones,  were studied e.g., in \cite{simple_comm,simple,flaska,jezek,zumbragel,simple_zero}
and, currently, they are of interest due to their possible applications in cryptography \cite{durcheva,maze}.

A complete classification of commutative congruence-simple semirings was done in \cite{simple_comm}. Finite congruence-simple semirings were classified in \cite{zumbragel} (up to a special case of additively idempotent semirings with a bi-absorbing element). 
The purpose of the present short note is to obtain basic classification of congruence-simple semirings possessing multiplicatively absorbing element and no non-trivial nilpotents (Theorem \ref{6.1} and Corollary \ref{commutative}).

\section{Preliminaries}

A \emph{semiring} is a non-empty set equipped with two associative binary operations (usually denoted as addition and multiplication) such that the addition is commutative and the multiplication distributes over the addition from both sides.

A non-empty subset $I$ of $S$ is an \emph{ideal} (\emph{bi-ideal}, resp.) of $S$ if $(I+I)\cup SI\cup IS\sub I$ ($(S+I)\cup SI\cup IS\sub I$, resp.). A semiring $S$ is called \emph{congruence-simple} if $S$ has just two congruences and \emph{ideal-simple} (\emph{bi-ideal-simple}, resp.) if $|S|\geq 2$ and $I=S$ whenever $I$ is an ideal (bi-ideal, resp.) containing at least two elements. 

A semiring $S$ is called \emph{additively idempotent} if $x+x=x$ for every $x\in S$ and \emph{additively cancellative} if $a+c\neq b+c$ for all $a,b,c\in S$ such that $a\neq b$.

An element $w\in S$ is called 
\begin{itemize}
\item \emph{multiplicatively (additively, resp.) absorbing} if $Sw=\{w\}=wS$ ($S+w=\{w\}$, resp.);
\item \emph{multiplicatively (additively, resp.) neutral} if $xw=x=wx$ ($x+w=x$, resp.) for every $x\in S$; 
\item \emph{bi-absorbing} if it is both multiplicatively and additively absorbing (such an element will be denoted by $o_S$);
\item a \emph{zero} if it is multiplicatively absorbing and additively neutral (such an element will be denoted by $0_S$).
\end{itemize}

Let $S$ be a semiring possessing a multiplicatively absorbing element $w$. An element $x\in S$ is called \emph{nilpotent} if $x^n=w$ for some $n\in\N$. Such an element $x$ is non-trivial if $x\neq w$.

\section{Two-element semirings with multiplicatively absorbing element}\label{two-elements}

One can check quite easily that (up to isomorphism) there exist just eight two-element semirings possessing a multiplicatively absorbing element. And these are:

\begin{center}
\begin{tabular}[t]{cc}%
\tabulka{$\mathbb{T}_1$}{$a$}{$b$}{$b$}{$a$}{$a$}{$a$}{$a$}{$a$} &
\tabulka{$\mathbb{T}_2$}{$a$}{$b$}{$b$}{$a$}{$a$}{$a$}{$a$}{$b$}\\
&\\[6pt]
\tabulka{$\mathbb{T}_3$}{$a$}{$b$}{$b$}{$b$}{$a$}{$a$}{$a$}{$a$} &
\tabulka{$\mathbb{T}_4$}{$a$}{$b$}{$b$}{$b$}{$a$}{$a$}{$a$}{$b$}\\
&\\[6pt]
\tabulka{$\mathbb{T}_5$}{$a$}{$a$}{$a$}{$b$}{$a$}{$a$}{$a$}{$a$} &
\tabulka{$\mathbb{T}_6$}{$a$}{$a$}{$a$}{$b$}{$a$}{$a$}{$a$}{$b$} \\

&\\[6pt]
\tabulka{$\mathbb{T}_7$}{$a$}{$a$}{$a$}{$a$}{$a$}{$a$}{$a$}{$a$} &
\tabulka{$\mathbb{T}_8$}{$a$}{$a$}{$a$}{$a$}{$a$}{$a$}{$a$}{$b$}
\end{tabular}
\end{center}

\bigskip
Let us notice, that the semirings $\mathbb{T}_1$, $\mathbb{T}_2$, $\mathbb{T}_3$ and $\mathbb{T}_4$ have a zero element $a=0_S$ and the semirings $\mathbb{T}_5$, $\mathbb{T}_6$, $\mathbb{T}_7$ and $\mathbb{T}_8$ have a bi-absorbing element $a=o_S$. The only cases that are rings are $\mathbb{T}_1$ and $\mathbb{T}_2$.

Of course, every two-element semiring is both congruence-simple and ideal-simple.

\section{Various useful congruences of semirings}

Throughout this section, let $S$ be a semiring.

\begin{remark}\label{3.1}
 The following binary relations $(\sub S\times S)$ are congruences of $S$ (easy to check directly):
 \begin{enumerate}
  \item $\alpha_I$, $I$ being an ideal of $S$:
  $(x,y)\in\alpha_I$ $\Leftrightarrow$  $x+a=y+b$ for some $a,b\in I$;
  \item $\beta_J$, $J$ being a bi-ideal of $S$:
  $\beta_J=(J\times J)\cup id_{S}$;
  \item $\gamma_n$, $n$ being a positive integer: $(x,y)\in\gamma_n$ $\Leftrightarrow$  $nx=ny$;
 \item $\delta$: $(x,y)\in\delta$ $\Leftrightarrow$ $2^ix=y+u$ and $2^iy=x+v$ for some $i\in\N_0$ and $u,v\in S$.
\end{enumerate}
\end{remark}

\begin{lemma}\label{3.0}
 Let $w\in S$ be a multiplicatively absorbing element. Then $w+w=w$. 
\end{lemma}
\begin{proof}
 Obviously, $w=w(w+w)=w^2+w^2=w+w$.
\end{proof}

\begin{lemma}\label{3.2}
 Let $I$ be an ideal of $S$. Then:
 \begin{enumerate}
  \item[(i)] $I\times I\sub \alpha_I$.
  \item[(ii)] $\alpha_I=id_S$ if and only if $I=\{0_S\}$.
 \end{enumerate}
\end{lemma}
\begin{proof}
 (i) Obvious.
 
 (ii) It is apparent that $I=\{0_S\}$ implies $\alpha_I=id_S$. Conversely, if $\alpha_I=id_S$, then $I=\{w\}$ is a one-element ideal and it follows immediately that $w$ is multiplicatively absorbing in $S$. On the other hand, by Lemma \ref{3.0}, $x+w+w=x+w$, for every $x\in S$. Hence $(x+w,x)\in\alpha_I=id_S$ and therefore $x+w=x$. Thus, $w=0_S$.
\end{proof}

\begin{lemma}\label{3.3}
 Let $I$ be a bi-ideal of $S$. Then:
 \begin{enumerate}
  \item[(i)] $\beta_I=id_S$ if and only if $I=\{o_S\}$.
  \item[(ii)] $\beta_I=S\times S$ if and only if $I=S$.
  \item[(iii)] $\beta_I\sub\alpha_I$.
  \item[(iv)] $\beta_I=\alpha_I$ if and only if $I=S$.
  \item[(v)] If $S$ has an additively absorbing element, then $\alpha_I=S\times S$.
 \end{enumerate}
\end{lemma} 
\begin{proof}
(i),(ii),(iii) and (v) are easy to verify. To show (iv), assume that $\beta_I=\alpha_I$. If $x,y\in S$, $x\neq y$ and $w\in I$ then $x+(y+w)=y+(x+w)$. As $x+w,y+w\in I$, we have that $(x,y)\in\alpha_I=\beta_I$. Hence $x,y\in I$ and we obtain that $I=S$.  
\end{proof}

\begin{lemma}\label{3.4}\noindent
The following hold:
\begin{enumerate}
\item[(i)] If $\gamma_2=id_S$ and $\delta=S\times S$ then the semiring $S$ is additively cancellative.
\item[(ii)] If $\gamma_2=S\times S$ and $\gamma_3=id_S$ then $S$ is a ring of characteristic $2$.
\item[(iii)] If $\gamma_2=S\times S=\gamma_3$ then $o_S\in S$ and $nx=o_S$ for all $x\in S$ and $n\geq 2$.
\item[(iv)]  $\delta=id_S$ if and only if the semiring $S$ is additively idempotent.
\end{enumerate}
\end{lemma}
\begin{proof}
(i) Indeed, assume that $x+z=y+z$ for some $x,y,z\in S$. From $\delta=S\times S$ we have that  $2^i y=z+v$ for some  $v\in S$ and $i\in\N_0$. Hence $y+2^iy=y+z+v=x+z+v=x+2^iy$. If $i\geq 1$, then $2(y+2^{i-1}y)=y+y+2^iy=y+x+2^iy=x+x+2^iy=2(x+2^{i-1}y)$. By $\gamma_2=id_S$, it follows that $y+2^{i-1}y=x+2^{i-1}y$ and, by induction, we obtain that $y+2^0y=x+2^0y$. Hence $2y=x+y$ and, similarly, $2x=x+y$. Thus $2y=2x$ and $y=x$. Therefore, $S$ is additively cancellative.

(ii)  From $\gamma_2=S\times S$ it follows that $2x=2y(=w)$ for all $x,y\in S$. Hence $w$ is multiplicatively absorbing and $2w=w$, by Lemma \ref{3.0}. Furthermore, $3x=x+w=x+w+w=3x+w=9x$. Since $\gamma_3=id_S$, we obtain that $x=3x=x+w$ and $w=0_S$. Hence $S$  is a ring of characteristic $2$.

(iii) As in (ii), from $\gamma_2=S\times S$ it follows that there is a unique multiplicatively absorbing element $w\in S$ such that $w=2x$ for all $x\in S$. Similarly, from $\gamma_3=S\times S$ we obtain that $w'=3x$ for all $x\in S$ and some multiplicatively absorbing element $w'\in S$. By the uniqueness, $w=w'$. Hence $w+x=2x+x=3x=w$ for every $x\in S$ and therefore $w=o_S$. The rest is obvious. 

(iv)  We have $2a=a+a$ and $2^2 a=2a+2a$ for every $a\in S$. Hence $(a,2a)\in\delta$. Consequently, $S$ is additively idempotent, provided that $\delta=id_S$. 

On the other hand, if $S$ is additively idempotent and $(x,y)\in\delta$ for some $x,y\in S$, then $2^ix=y+u$, $2^iy=x+v$ for some $u,v\in S$ and therefore $x=y+u$, $y=x+v$. Hence $x=x+x=x+y+u=x+y+y+u=x+y+x=x+y$ and similarly $y=y+x$. Therefore $x=x+y=y$ and $\delta=id_S$.
\end{proof}

\section{Semirings with multiplicatively absorbing element and without nilpotent elements}

In this section, let $S$ be a semiring possessing a multiplicatively absorbing element $w$ and let $x^2\neq w$ for every $x\in S$, $x\neq w$. Obviously, this condition equivalently means that $S$ has no non-trivial nilpotent elements.

\begin{lemma}\label{4.1}
The following hold:
 \begin{enumerate}
  \item[(i)] $S+w$ is a bi-ideal of $S$.
  \item[(ii)] $|S+w|=1$ if and only if $w=o_S$.
  \item[(iii)] $S+w=S$ if and only if $w=0_S$.
 \end{enumerate}
\end{lemma}
\begin{proof}
 (i) and (ii) are easy to check. 
 
 (iii) Assume that $S+w=S$. For every $a\in S$ there is $b\in S$ such that $b+w=a$. Hence, by Lemma \ref{3.0},  $a+w=b+w+w=b+w=a$ and therefore $w=0_S$. The opposite implication is obvious.  
\end{proof}

For every $a\in S$, put $(a:w)_{\ell}=\set{x\in S}{xa=w}$ and $(a:w)_{r}=\set{x\in S}{ax=w}$. Clearly, $w\in (a:w)_{\ell}\cap(a:w)_{r}$, $(a:w)_{\ell}$ is a left ideal of $S$ and $(a:w)_{r}$ is a right ideal of $S$.

\begin{lemma}\label{4.2}
 Let $a\in S$. Then:
 \begin{enumerate}
  \item[(i)] $(a:w)_{\ell}=(a:w)_{r}$ $(=(a:w))$ is an ideal of the semiring $S$.
  \item[(ii)] $(a:w)=\{w\}$ if and only if $ab\neq w\neq ba$ for every $b\in S$, $b\neq w$.
  \item[(iii)] $(a:w)=S$ if and only if $Sa=\{w\}=aS$.
 \end{enumerate}
\end{lemma}
\begin{proof}
 (i) It suffices to show that $(a:w)_{\ell}=(a:w)_r$. If $x\in (a:w)_{\ell}$, then $(ax)^2=axax=awx=w$. As $S$ has no non-trivial nilpotents, we obtain that $ax=w$ and $x\in (a:w)_r$. The rest is clear.

(ii) and (iii) are obvious.
\end{proof}

\begin{lemma}\label{4.3}
 Let $a\in S$. If $\alpha_{(a:w)}=id_S$ (see Remark $\ref{3.1}(1)$) then $w=0_S$, $ab\neq 0_S\neq ba$ for every $b\in S$, $b\neq w$, and either $|S|=1$ or $a\neq 0_S$.
\end{lemma} 
\begin{proof}
 Let $a\in S$ and $\alpha_{(a:w)}=id_S$. By Lemma \ref{3.2}(ii), $w\in(a:w)=\{0_S\}$ and therefore $w= 0_S$. By \ref{4.2}(ii), we have that $ab\neq 0_S\neq ba$ for every $b\in S$, $b\neq w$. The rest is obvious.
\end{proof}

\begin{lemma}\label{4.4}
 Assume that $w=0_S$. Let $a\in S$. Then:
 \begin{enumerate}
  \item[(i)] $0_S\in(a:0_S)$ and the ideal $(a:0_S)$ is a block of the congruence $\alpha_{(a:0_S)}$.
  \item[(ii)] If $\alpha_{(a:0_S)}=id_S$ then $ab\neq 0_S\neq ba$ for every $b\in S$, $b\neq 0_S$.
  \item[(iii)] If $\alpha_{(a:0_S)}=S\times S$ then $a=0_S$.
 \end{enumerate}
\end{lemma}
\begin{proof}
(i) Let $(x,0_S)\in\alpha_{(a:0_S)}$ for some $x\in S$. Then there are $u,v\in(a:0_S)$ such that  $x+u=0_S+v$ and we obtain that $ax=ax+0_S=ax+au=a(x+u)=a(0_S+v)=a0_S+av=0_S$ and therefore $x\in(a:0_S)$. Hence, by  Lemma \ref{3.2}(i), it follows that  $(a:0_S)$ is a block of the congruence $\alpha_{(a:0_S)}$.

(ii) Follows immediately from Lemma \ref{4.3}.

(iii) If $\alpha_{(a:0_S)}=S\times S$ then, by (i), $(a:0_S)=S$. Hence $a\in(a:0_S)$ and therefore $a^2=0_S$ and, as there are no non-trivial nilpotents in $S$, we have that $a=0_S$.
\end{proof}

\begin{lemma}\label{4.5}
 Assume that $w=o_S$. Let $a\in S$. Then:
 \begin{enumerate}
  \item[(i)] $(a:o_S)$ is a bi-ideal of $S$.
  \item[(ii)] $\alpha_{(a:o_S)}=S\times S$.
  \item[(iii)] If $\beta_{(a:o_S)}=id_S$ (see Remark $\ref{3.1}(2)$), then $ab\neq o_S\neq ba$ for every $b\in S$, $b\neq o_S$.
  \item[(iv)] If $\beta_{(a:o_S)}=S\times S$, then $Sa=\{o_S\}=aS$.
 \end{enumerate}
\end{lemma}
\begin{proof}
 (i) is easy to see. To show (ii), we obviously have that $b+o_S=o_S+o_S$ for every $b\in S$. Hence $(b,o_S)\in (a:o_S)$ and therefore $\alpha_{(a:o_S)}=S\times S$.
 
 (iii) and (iv) follow from Lemma \ref{3.3}(i) and (ii). 
\end{proof}

The following result is a moderate generalization of \cite[Theorem 2.2]{cornish} where it was proved that a congruence-simple semiring with a zero and no non-trivial nilpotent elements has no non-trivial zero-divisors.

\begin{theorem}\label{4.6}
 Assume that the semiring $S$ is congruence-simple (and $S$ has a multiplicatively absorbing element $w$ such that $x^2\neq w$ for every $x\in S$, $x\neq w$.) Then:
 \begin{enumerate}
  \item[(i)] Either $w= 0_S$ or $w=o_S$.
  \item[(ii)] $ab\neq w$ for all $a,b\in S\setminus\{w\}$.
 \end{enumerate}
\end{theorem}
\begin{proof}
 (i) The set $I=S+w$ is a bi-ideal and $\beta_I=(I\times I)\cup id_S$ is a congruence of $S$. Hence either $\beta_I=id_S$ or $\beta_I=S\times S$. By Lemma \ref{3.3}, we obtain that $I=\{o_S\}$ or $I=S$. Therefore, by Lemma \ref{4.1}, $w=0_S$ or $w=o_S$.
 
 (ii) If $w=0_S$ and $a\in S\setminus\{0_S\}$ then, by Lemma \ref{4.4}(iii), we have that $\alpha_{(a:0_S)}\neq S\times S$. Hence $\alpha_{(a:0_S)}= id_S$ and, by Lemma \ref{4.4}(ii), it follows that $ab\neq w$ for all $a,b\in S\setminus\{w\}$. 
 
 If $w=o_S$ and $a\in S\setminus\{o_S\}$, then $a^2\neq o_S$. Therefore, by Lemma \ref{4.5}(iv), we have that $\beta_{(a:0_S)}\neq S\times S$. Hence $\beta_{(a:0_S)}= id_S$ and, by Lemma \ref{4.5}(iii), it follows that $ab\neq w$ for all $a,b\in S\setminus\{w\}$.
\end{proof}

\begin{proposition}\label{5.3}
 Assume that the semiring $S$ is a congruence-simple and has a bi-absorbing element $o_S$.  Then just one of the following two cases takes place:
 \begin{enumerate}
  \item $S\cong\mathbb{T}_8$;
  \item $S+S=S$ and $ab\neq o_S$ for all $a,b\in S\setminus\{o_S\}$.
 \end{enumerate}
\end{proposition}
\begin{proof}
 By Theorem \ref{4.6}, $TT\sub T$ for $T=S\setminus\{o_S\}$. Assume that the case (2) does not hold. Then $S+S\neq S$ and, as $S+S$ is a bi-ideal, we have that $S+S=\{o_S\}$. The relation $\varrho=(T\times T)\cup \{(o_S,o_S)\}$ is now a congruence of $S$ and $\varrho\neq S\times S$. Thus $\varrho=id_S$ and therefore $|S|=2$ and $S\cong\mathbb{T}_8$.
\end{proof}

\section{The main result}

In this section we provide a classification of congruence-simple semirings with a multiplicatively absorbing element and without non-trivial nilpotent elements.

\begin{theorem}\label{6.1}
 Let $S$ be a congruence-simple semiring possessing a multiplicatively absorbing element $w$ such that $x^2\neq w$ for every $x\in T=S\setminus\{w\}$. Then just one of the following five cases takes place:
 \begin{enumerate}
  \item $S\cong\mathbb{T}_4$ or $\mathbb{T}_8$;
  \item $S$ is a finite field;
  \item $S$ is an infinite simple ring without zero-divisors;
  \item $w=o_S$, $S$ is additively idempotent and $TT\sub T$;
  \item $w=o_S$, $S$ is infinite, $S+S=S$, $2x=o_S$ for every $x\in S$, $TT\sub T$ and, for every $y\in T$, the powers $y,y^2,y^3,\dots$ are pair-wisely different.
  
 \end{enumerate}
\end{theorem} 
\begin{proof}
If $S$ has two elements then, according to Section \ref{two-elements}, it is isomorphic to one of the semirings $\mathbb{T}_4$ or $\mathbb{T}_8$ (case (1)) or $\mathbb{T}_2$ (case (2)) or $\mathbb{T}_6$ (case (4)). We can assume therefore that $|S|\geq 3$. Due to Theorem \ref{4.6}, $TT\sub T$ and either $w=0_S$ or $w=o_S$.

Firstly, let $w=0_S$ (i.e., $w$ is additively neutral and multiplicatively absorbing). Put $R=\set{a\in S}{(\exists c\in S)\ a+c=0_S}$. Clearly, $0_S\in R$ and $R$ is an ideal of $S$. Moreover, $R$ is a ring and $a+b\in R$ implies $a,b\in R$ for every $a,b\in S$. Now, consider the congruence $\alpha_{R}$ (see Remark \ref{3.1}(1)).
 
 Assume that $\alpha_R=id_S$. By Lemma \ref{3.2}(ii), we have that $R=\{0_S\}$ and it follows that the set $T$ is a subsemiring of $S$. Now, $\varrho=(T\times T)\cup \{(0_S,0_S)\}$ is a congruence of $S$ and $\varrho\neq S\times S$. Hence $\varrho=id_S$ and $|S|=2$, a contradiction with $|S|\geq 3$. 
 
 Next, assume that $\alpha_R\neq id_S$. Then $\alpha_R=S\times S$ and therefore for every $a\in S$ there are $u,v\in R$ such that $a+u=0_S+v\in R$. By the properties of $R$, it follows that $a\in R$ and thus $R=S$. Therefore $S$ is a simple ring and, as $ab\neq 0_S$ for all $a,b\in T$, the ring $S$ is without zero-divisors. If $S$ is infinite, we have obtained the case (3).  Assume therefore that $S$ is finite. Since $S$ has no zero-divisors, the multiplicative semigroup $T(\cdot)$ is cancellative and finite. Hence $T(\cdot)$ is a group and, by the well known Wedderburn theorem,  $S$ is a (commutative) field (the case (2)).
 
 Secondly, let $w=o_S$ (i.e., $w$ is both additively  and multiplicatively absorbing). We see immediately that $S$ is not additively cancellative.  
 If $S$ is additively idempotent we obtain the case (4). 
 
 Assume, henceforth, that $S$ is not additively idempotent. By  Lemma \ref{3.4}(i) and (iv) we have that $\delta=S\times S$ and $\gamma_2\neq id_S$. Hence $\gamma_2= S\times S$ and, by Lemma \ref{3.4}(ii), $\gamma_3\neq id_S$. Therefore $\gamma_3= S\times S$ and, by Lemma \ref{3.4}(iii), we conclude that $2x=o_S$ for every $x\in S$. By Proposition \ref{5.3}, $S+S=S$. The fact that the powers $y^n$, $n\geq 1$, are pair-wisely different is proved in \cite[Lemma 8.8]{zeropotent}. We have obtained the last case (5) which concludes our proof. 
\end{proof}

\begin{example}\label{6.3}

(i) Let $G(\cdot)$ be a semigroup, $o\notin G$ and $V(G)=G\cup\{o\}$. Put $x\cdot o = o=o\cdot x$, $x+x=x$ and $x+y=o$ for all $x,y\in V(G)$, $x\neq y$. If $G$ is cancellative (i.e., $ac\neq bc$ and $ca\neq cb$ for all $a,b,c\in G$ such that $a\neq b$) then $V(G)(+,\cdot)$ becomes an additively idempotent semiring, the element $o$ is bi-absorbing and $x^2\neq o$ for every $x\in G$. 

Assume, moreover, that $G$ is a simple semigroup (i.e., $GaG=G$ for every $a\in G$). It is not complicated to check that the semiring $V(G)$ is then both congruence- and ideal-simple. Clearly, the semiring $V(G)$ is of type (4) in Theorem \ref{6.1}. This example, moreover, shows that the additive analogue of Theorem \ref{4.6}(ii) is not true. 

(ii) A natural example of a simple cancellative semigroup in (i) is a group. Let us notice that also every finite cancellative semigroup is a group. Thus, if the semiring $V(G)$ is finite then $G$ has to be a group. On the other hand, there exist (infinite) simple cancellative semigroups that are not groups (e.g., the multiplicative matrix semigroup $\set{(\begin{smallmatrix}
                a & b   \\                                                                                                                                                                                                                     0& 1\end{smallmatrix}
)}{a,b\in\R^+}$).   
\end{example}

\begin{corollary}\label{commutative}
  Let $S$ be a congruence-simple semiring possessing a multiplicatively absorbing element $w$ such that $x^2\neq w$ for every $x\in S\setminus\{w\}$. If $S$ is commutative then just one of the following three cases takes place:
 \begin{enumerate}
  \item $S\cong\mathbb{T}_4$ or $\mathbb{T}_8$;
  \item $S$ is a field;
  \item $S\cong V(G)$, where $G$ is a commutative group (see Example $\ref{6.3}$).
 \end{enumerate}
\end{corollary}
\begin{proof}
 It follows immediately from Theorem \ref{6.1} and the classification of commutative congruence-simple semirings in \cite[Theorem 10.1]{simple_comm}.
\end{proof}

\begin{remark}\label{6.2}
 (i) Let us note that among the list in Theorem \ref{6.1} the \emph{only} semiring with a zero that is not a ring is $\mathbb{T}_4$. This semiring is additively idempotent.

 (ii) There are well known examples of rings of type (3) in Theorem \ref{6.1}. Every division ring is a simple ring without zero-divisors. On the other hand, examples of 
 simple rings without zero-divisors that are not division rings are also not rare (see, e.g., \cite{cohn1} and \cite{cohn2}).
 
 (iii) No example of a semiring of type (5) in Theorem \ref{6.1} is known to the authors of the present brief note.
\end{remark}

\begin{remark}\label{6.4}
Finally, let us note that in \cite[Theorem 2.1]{cornish} there is constructed an example of a \emph{subdirectly irreducible} semiring $S$ with a zero that contains no non-trivial nilpotents and yet \emph{has} non-trivial divisors of zero. This semirings $S$ has five elements, is commutative, has a unity and is both multiplicatively and additively idempotent.
\end{remark}

\section{A conjecture}

In the last section we propose a conjecture on finite semirings from Theorem \ref{6.1}. First, let us recall a few results on congruence-simple semirings in  \cite{jezek,zumbragel}.

 Let $L(+)$ be a semilattice with the greatest element $1$ and $|L|\geq 2$. Denote by $End_1(L)(+,\cdot)$ the semiring of all endomorphisms of the semilattice $L$ preserving the element $1$. The semiring operations on $End_1(L)$ are defined in the usual way, i.e.,  for all $r,s\in End_1(L)$ and every $x\in L$ we have $(r+s)(x)=r(x)+s(x)$ and $(r\cdot s)(x)=r(s(x))$.  The semiring $End_1(L)$ is additively idempotent (with an induced natural ordering) and has a bi-absorbing element $o$, where $o(x)=1$ for every $x\in L$. 

Further, denote by $Y(L)$ the set  of all endomorphisms from $End_1(L)$ with the range at most $2$ and put  $\mathbf{G}(L)=\set{f\in End_{1}(L)}{\exists g\in Y(L)\ \ g\leq f}$.

\begin{remark}\label{7.1}
In \cite[Theorem 2.2]{jezek} it was shown that every subsemiring $S\sub End_{1}(L)$ such that $Y(L)\sub S\sub \mathbf{G}(L)$ is congruence-simple (and contains the bi-absorbing element $o$). Notice also, that the set $\mathbf{G}(L)$ itself is a subsemiring of $End_{1}(L)$.

Now, the semiring $S$ has \emph{no non-trival nilpotent elements} if and only if $|L|=2$ (and, consequently, if $S\cong\mathbb{T}_6$).

Indeed, if $|L|=2$ then, obviously, $S=\{o,i\}$, where $i$ is the identical map on $L$, and $S$ has therefore no non-trival nilpotent elements.

Assume now that $|L|\geq 3$. Then there are $a,b\in L\setminus\{1\}$ such that $b\not\leq a$. Let $e: L\to L$ be a map defined as $e(x)=b$ if $x\leq a$ and $e(x)=1$ otherwise, for every  $x\in L$. It is easy to check that $e\in Y(L)$  and, consequently, $e\in S$. Since $e(x)\geq b$ for every $x\in L$ it follows that $e(e(x))\geq e(b)=1=o(x)$. As $o$ is the greatest element of $S$, we obtain that $e^2=o$ and $e$ is a non-trivial nilpotent element in $S$. 
\end{remark}

\begin{remark}\label{7.2}
 (i) In \cite{zumbragel} the following construction was studied. Let $S$ ($S'$, resp.) be semirings with at least two elements and with a bi-absorbing element $o$ ($o'$, resp.). Let $S\times S'$ be the product of these semiring (with the usual component-wise operations) and with a bi-absorbing element $(o,o')$. The set $$I=\set{(s,s')\in S\times S'}{s=o\ \  \text{or}\ \ s'=o'}$$ is a bi-ideal of $S\times S'$. By Remark \ref{3.1}(2) the relation $\beta_I=(I\times I)\cup id_{S\times S'}$ is a congruence on the semiring $S\times S'$, and we may define a \emph{box-product} of the semirings $S$ and $S'$ as follows: $$S\boxtimes S'=(S\times S')/\beta_I\ .$$
 
 Obviously, the semiring $S\boxtimes S'$ has a bi-absorbing element $(o,o')/\beta_I$ and it is easy to check that if $S\cong\mathbb{T}_6$ then $S\boxtimes S'\cong S'$.
 
Notice also, that if the semiring $S$ has a non-trivial nilpotent element then we may easily find a non-trivial nilpotent element in $S\boxtimes S'$.

(ii) In \cite[Section 8]{zumbragel} it was conjectured that every finite additively idempotent congruence-simple semiring with a bi-absorbing element is isomorphic to $S\boxtimes V(G)$ for some semiring $S$ from Remark \ref{7.1}  and  some finite group $G$   (see Example \ref{6.3}).

Let us note that, according to part (i) and  Remark \ref{7.1}, the semiring   $S\boxtimes V(G)$ has no non-trivial nilpotent elements if and only if $S\cong \mathbb{T}_6$ (and, consequently, if  $S\boxtimes V(G)\cong V(G)$). 
\end{remark}

As a conclusion, the conjecture on the structure of finite congruence-simple semirings in \cite{zumbragel} implies, by Theorem \ref{6.1} and Remark \ref{7.2}, the following conjecture.

\begin{conjecture}
  Let $S$ be a congruence-simple semiring possessing a multiplicatively absorbing element $w$ such that $x^2\neq w$ for every $x\in S\setminus\{w\}$. If $S$ is finite then just one of the following three cases takes place:%
 \begin{enumerate}
  \item $S\cong\mathbb{T}_4$ or $\mathbb{T}_8$;
  \item $S$ is a finite field;
   \item $S\cong V(G)$, where $G$ is a finite group (see Example $\ref{6.3}$).
 \end{enumerate}
\end{conjecture}


\end{document}